\documentclass[11pt]{article}

\usepackage{amsthm}
\usepackage{amssymb}
\usepackage{hyperref}
\usepackage{mathrsfs}
\usepackage{color}
\usepackage{amsmath}
\usepackage{diagbox}

\def\sttf2#1#2{\left[\!\!\left[#1\atop#2\right]\!\!\right]}
\def\stf3f#1#2{\left[\!\!\left[\!\!\left[#1\atop#2\right]\!\!\right]\!\!\right]}
\def\stff4#1#2{\left[\!\!\left[\!\!\left[\!\!\left[#1\atop#2\right]\!\!\right]\!\!\right]\!\!\right]}
\def\stss2#1#2{\left\{\!\!\left\{#1\atop#2\right\}\!\!\right\}}

\newtheorem{theorem}{Theorem}

\newtheorem{Cor}{Corollary}
\newtheorem{Lem}{Lemma}

\allowdisplaybreaks

\begin{document}

\title{Tribonacci properties of identities, matrices, and determinants}

\author{
Takao Komatsu
\\
\small Institute of Mathematics\\[-0.8ex]
\small Henan Academy of Sciences\\[-0.8ex]
\small Zhengzhou 450046 China\\[-0.8ex]
\small \texttt{komatsu@zstu.edu.cn}\\
\small and\\
\small Department of Mathematics\\[-0.8ex]
\small Institute of Science Tokyo\\[-0.8ex]
\small 2-12-1 Ookayama, Meguro-ku\\[-0.8ex]
\small Tokyo 152-8551 Japan\\[-0.8ex]
\small \texttt{komatsu.t.al@m.titech.ac.jp}\\\\
Tengfei Shen\\
\small School of Mathematics\\[-0.8ex]
\small Northwest University\\[-0.8ex]
\small Research Center for Number Theory and Its Applications\\[-0.8ex]
\small Xi'an 710127 China\\[-0.8ex]
\small \texttt{stf13675206736@163.com}
}

\date{
}

\maketitle

\begin{abstract}
This paper considers the properties of Tribonacci numbers on identities, matrices, and determinants. In the first front part, we obtain several symmetric identities of Tribonacci numbers by a matrix-based approach and binomial inversion technique. In the core section of the latter half, we present a determinant representation of Tribonacci numbers in a slightly modified Toeplitz--Hessenberg form derived from Bell polynomials.


\noindent
{\bf MR Subject Classifications:} 11B39, 11B83, 11C20, 05A15, 05A19, 11B37.
\end{abstract}

\section{Introduction}

The sequence of Tribonacci numbers is believed to have first been investigated by M. Feinberg in 1963.
Although there are some variations in the initial values adopted by different authors, here we define the Tribonacci numbers $T_n$ as
\begin{equation}
T_n=T_{n-1}+T_{n-2}+T_{n-3}\quad(n\ge 3),\quad T_0=0,\,T_1=T_2=1\,.
\label{def:tribo}
\end{equation}
Since Tribonacci numbers are considered a type of extension of Fibonacci numbers, they share many properties similar to those of Fibonacci numbers.
For instance, just as Fibonacci numbers can be derived by using the rising diagonals of Pascal's triangle, Tribonacci numbers can be derived by adding up the elements on the rising diagonals of a similar triangular array.

Because of their more complex structural properties, combinatorial identities of Tribonacci numbers remain less explored compared to the Fibonacci case. Especially, identities involving double binomial coefficients and their symmetric forms via matrix inversion remain largely unaddressed.

In the first part of this paper, we address this gap by deriving a series of new combinatorial identities for Tribonacci numbers, using matrix methods and binomial inversion techniques. Our matrix-based approach provides a unified framework for deriving symmetric identities for linear recurrence sequences, which is applicable to Fibonacci, Tribonacci, and higher-order sequences.

While Fibonacci numbers have a domino model, Tribonacci numbers possess a similar domino model that incorporates triminoes.
While we can experience the beauty and power of the $Q$-matrix for extracting properties of Fibonacci numbers and polynomials, we can find a similar matrix for Tribonacci numbers and polynomials.  Consequently, for example, we can have
$$
\left|\begin{array}{ccc}
T_{n+2}&T_{n+1}&T_n\\
T_{n+1}&T_n&T_{n-1}\\
T_n&T_{n-1}&T_{n-2}
\end{array}\right|=-1\,.
$$
For the fundamental properties of Tribonacci numbers, see \cite[Ch. 49]{Koshy2} (see also \cite{Koshy1}).

The study of determinants holds fundamental importance; in particular, Fibonacci determinants serve as exemplary illustrations of their capacity to reveal elegant interconnections between linear algebra and number theory. A substantial body of literature addresses Fibonacci matrices or determinants, that is, matrices or determinants whose entries consist of Fibonacci numbers. Nevertheless, explicit determinant representations that directly express Fibonacci numbers remain relatively scarce. In response, several results have been established to express Fibonacci numbers and their companion sequences (see \cite{KKL22,Ko20b}) through the application of Cameron's operator \cite{Cameron}. More recently, these findings have been extended and generalized in \cite{Ko25a} to encompass Fibonacci and related polynomials. Additional related results are reported in \cite{KT10,KTH10}.

In a related development, \cite{GS20} presents multiple determinantal expressions for certain families of Toeplitz--Hessenberg matrices populated with Tribonacci numbers. Although determinantal expressions for Tribonacci numbers themselves have received limited attention, analogous determinant forms have been employed in \cite{KY26} to provide determinant representations of generalized Tribonacci numbers. In the core section of the latter half, we present a determinant representation of Tribonacci numbers in a slightly modified form related to Bell polynomials.

\section{Tribonacci identities}

While Fibonacci numbers possess many interesting identities, Tribonacci numbers are also known to have several identities, though not as numerous.
For example, while Fibonacci numbers satisfy
$$
F_{m+n}=F_{m+1}F_n+F_m F_{n-1}\quad({\rm \cite[(20.2)]{Koshy1}})\,,
$$
Tribonacci numbers satisfy
$$
T_{m+n}=T_{m+1}T_n+T_m T_{n-1}+T_m T_{n-2}+T_{m-1}T_{n-1}\quad({\rm \cite[Corollary 49.4 (1)]{Koshy2}})\,.
$$
While Fibonacci numbers are related to quadratic polynomials and binomial coefficients, Tribonacci numbers are related to cubic polynomials and trinomial coefficients (see, e.g., \cite{TongronK22}).  

More generally, for fixed real numbers with $(u,v,w)\ne(0,0,0)$ and $(a,b,c)\ne(0,0,0)$, let $\mathcal T_n^{(a,b,c)}$ be a generalized Tribonacci number, defined by
\begin{equation}
\mathcal T_n^{(a,b,c)}=u\mathcal T_{n-1}^{(a,b,c)}+v\mathcal T_{n-2}^{(a,b,c)}+w\mathcal T_{n-3}^{(a,b,c)}\quad(n\ge 3)
\label{def:ggtribo}
\end{equation}
with $\mathcal T_0=a$, $\mathcal T_1=b$, and $\mathcal T_2=c$.
When $(u,v,w)=(1,1,1)$ and $(a,b,c)=(0,1,1)$, $T_n=\mathcal T_n^{(0,1,1)}$ is the original Tribonacci number given in (\ref{def:tribo}). When $(u,v,w)=(1,0,1)$, $\mathcal T_n^{(0,1,1)}$ yields the Padovan sequence.

In this and next sections, assume that $(a,b,c)=(0,1,1)$.  For convenience, put $\mathsf T_{n}:=\mathsf T_n^{(0,1,1)}$.
In a similar manner to \cite{TongronK22},   
for integers $n_{1},n_{2},n_{3}$, setting  
\begin{align*}
\mathsf T(n_{1},n_{2},n_{3}):=T_{n_{1}+n_{2}+n_{3}}\,, 
\end{align*}
we shall consider the properties.   

For integers $n$ and $k$ with $n\geq{2k}$, we have
\begin{align*}
\mathsf T(n,0,k)&=u\mathsf T(n,0,k-1)+v\mathsf T(n-1,0,k-1)+w\mathsf T(n-1,-1,k-1)\\
&=\sum_{i=0}^{1}u^{1-i}\sum_{j=0}^{i}v^{i-j}w^{j}\mathsf T(n-i,-j,k-1)\,. 
\end{align*}
By extending this, we get
\begin{align*}
&\mathsf T(n,0,k)\\
&= \sum_{i_1=0}^{1}\cdots\sum_{i_k=0}^{1}u^{1-(i_{1}+\cdots+i_{k})}
  \sum_{j_{1}=0}^{i_1}\cdots\sum_{j_k=0}^{i_k}v^{i_{1}+\cdots+i_{k}-(j_{1}+\cdots+j_{k})}w^{j_{1}+\cdots+j_{k}}\\
  &\mathsf T\!\left(n-(i_{1}+\cdots+i_{k}),\ -(j_{1}+\cdots+j_{k}),\ 0\right)\nonumber\\
&=\sum_{I=0}^{k}\sum_{i_{1}+\cdot\cdot\cdot+i_{k}=I}u^I\sum_{J=0}^{I}\sum_{j_{1}+\cdot\cdot\cdot+j_{k}=J}v^{I-J}w^J\mathsf T(n-I,-J,0)\nonumber\\
&=\sum_{I=0}^{k}\sum_{J=0}^{I}\binom{k}{I}\binom{I}{J}u^Iv^{I-J}w^J\mathsf T(n-I,-J,0).\nonumber
\end{align*}
By applying $\mathsf T(n,0,k)=T_{n+k}$, for integers $n$ and $k$ with $n\geq{2k}$, replacing $I$ with $i$, $J$ with $i$ for the general form, we get
$$
T_{n+k}=\sum_{i=0}^{k}\sum_{j=0}^{i}\binom{k}{i}\binom{i}{j}u^Iv^{i-j}w^jT_{n-i-j},
$$

By setting $uv=\alpha,\ wv^{-1}=\beta$, we have the following

\begin{theorem}
\begin{align*}
T_{n+k}&=\sum_{i=0}^{k}\sum_{j=0}^{i}\binom{k}{i}\binom{i}{j}\alpha^{i}\beta^{j}T_{n-i-j}\\
&=\sum_{u=0}^{2k}T_{n-u}\sum_{i=\left\lfloor\frac{u}{2}\right\rfloor}^{k}\binom{k}{i}\binom{i}{u-i}\alpha^{i}\beta^{u-i}\,.
\end{align*}
\label{theom1}
\end{theorem}
\begin{proof}
Since
\begin{align*}
&\sum_{i=0}^{k}\sum_{j=0}^{i}\binom{k}{i}\binom{i}{j}\alpha^{i}\beta^{j}\mathsf \mathsf T_{n-i-j}\\
&=\sum_{i+j=0}^{2k}\mathsf \mathsf T_{n-i-j}\binom{k}{i}\binom{i}{j}\alpha^{i}\beta^{j}\\
&=\sum_{u=0}^{2k}\mathsf \mathsf T_{n-u}\sum_{i=\left\lfloor\frac{u}{2}\right\rfloor}^{\max\{u,k\}}\binom{k}{i}\binom{i}{u-i}\alpha^{i}\beta^{u-i}\\
&=\sum_{u=0}^{2k}\mathsf \mathsf T_{n-u}\sum_{i=\left\lfloor\frac{u}{2}\right\rfloor}^{k}\binom{k}{i}\binom{i}{u-i}\alpha^{i}\beta^{u-i}\,,
\end{align*}
we have
\begin{align*}
T_{n+k}=\sum_{u=0}^{2k}\mathsf \mathsf T_{n-u}\sum_{i=\left\lfloor\frac{u}{2}\right\rfloor}^{k}\binom{k}{i}\binom{i}{u-i}\alpha^{i}\beta^{u-i}\,.
\end{align*}
\end{proof}

\section{Tribonacci matrices}

Consider the vector\ $\mathbf{T}:=( T_{n+k}, T_{n+k-1}, \cdots, T_{n})^T$.
By Theorem \ref{theom1}, we have
{\scriptsize
\begin{align}
&\begin{pmatrix} T_{n+k} \\ T_{n+k-1} \\ \vdots \\ T_{n+1} \\ T_n \end{pmatrix}
=
\begin{pmatrix}
\sum_{i=0}^{k} \binom{k}{i} \sum_{j=0}^{i} \binom{i}{j}\alpha^{i}\beta^{j} T_{n-(i+j)} \\[6pt]
\sum_{i=0}^{k-1} \binom{k-1}{i} \sum_{j=0}^{i} \binom{i}{j}\alpha^{i}\beta^{j} T_{n-(i+j)} \\[6pt]
\vdots \\[6pt]
\sum_{i=0}^{1} \binom{1}{i} \sum_{j=0}^{i} \binom{i}{j}\alpha^{i}\beta^{j} T_{n-(i+j)} \\[6pt]
\sum_{i=0}^{0} \binom{0}{i} \sum_{J=0}^{i} \binom{i}{j}\alpha^{i}\beta^{j} T_{n-(i+j)}
\end{pmatrix}\nonumber\\
&=
\begin{pmatrix}
\binom{k}{0}\alpha^{0} & \binom{k}{1}\alpha^{1} & \binom{k}{2}\alpha^{2} & \cdots & \binom{k}{k-1}\alpha^{k-1} & \binom{k}{k}\alpha^{k} \\
\binom{k-1}{0}\alpha^{0} & \binom{k-1}{1}\alpha^{1} & \binom{k-1}{2}\alpha^{2} & \cdots & \binom{k-1}{k-1}\alpha^{k-1} & 0 \\
\binom{k-2}{0}\alpha^{0} & \binom{k-2}{1}\alpha^{1} & \binom{k-2}{2}\alpha^{2} & \cdots & 0 & 0 \\
\vdots & \vdots & \vdots & \vdots & \vdots & \vdots \\
\binom{1}{0}\alpha^{0} & \binom{1}{1}\alpha^{1} & 0& \cdots & 0 & 0 \\
\binom{0}{0}\alpha^{0} & 0 & 0 & \cdots & 0 & 0
\end{pmatrix}
\begin{pmatrix}
\mathop\sum\limits_{j=0}^{0} \binom{0}{j}\beta^{j} T_{n-j} \\[6pt]
\mathop\sum\limits_{j=0}^{1} \binom{1}{j}\beta^{j} T_{n-(1+j)} \\[6pt]
\vdots \\[6pt]
\mathop\sum\limits_{j=0}^{k-1} \binom{k-1}{j}\beta^{j} T_{n-(k-1+j)} \\[6pt]
\mathop\sum\limits_{j=0}^{k} \binom{k}{j}\beta^{j} T_{n-(k+j)}
\end{pmatrix}\,.
\label{mat1}
\end{align}
}
Let $\widetilde{P}$ be the matrix on the left side of the right-hand side of (\ref{mat1}), which is the transform row-reversed matrix of lower Pascal matrix with parameter $\alpha$. Then we can derive that 
\begin{align*}
\widetilde{P}_{i,j}^{-1}=(-1)^{i+j-k}\binom{i}{k-j}\alpha^{-i},
\end{align*}
where $0\leq{i,j}\leq{k}$.
\medskip

Performing the transform of ${\widetilde{P}}^{-1}$\ on (\ref{mat1}), for $0\leq{i}\leq{k}$, we get
\begin{align}
\sum_{j=0}^{k}(-1)^{i+j-k}\binom{i}{k-j}\alpha^{-i}T_{n+k-j}=\sum_{t=0}^{i}\binom{i}{t}\beta^{t}T_{n-(i+t)}\,.
\end{align}
Since
\begin{align*}
\sum_{j=0}^{k}(-1)^{i+j-k}\binom{i}{k-j}\alpha^{-i}T_{n+k-j}&=\sum_{j=k-i}^k(-1)^{i+j-k}\binom{i}{k-j}\alpha^{-i}T_{n+k-j}\\
&=\sum_{t=0}^i(-1)^{i-t}\binom{i}{t}\alpha^{-i}T_{n+t}\,,
\end{align*}
we have the following

\begin{theorem}
\begin{align*}
\alpha^{-i}\sum_{t=0}^i(-1)^{i-t}\binom{i}{t}T_{n+i}=\sum_{t=0}^i\binom{i}{t}\beta^{t}T_{n-(i+t)}\,.
\end{align*}
\end{theorem}

Putting $n=2k$ and $i=k$, we get
$$
\alpha^{-i}\sum_{t=0}^k(-1)^{k-t}\binom{k}{t}T_{2 k+t}=\sum_{j=0}^k\binom{k}{j}T_{k-j}=\sum_{j=0}^k\binom{k}{j}\beta^{j}T_j\,.
$$
By using the binomial inversion, we have the following

\begin{Cor}
\begin{align}
\beta^{j}T_j&=\alpha^{-i}\sum_{u=0}^j(-1)^{j-u}\binom{j}{u}\sum_{t=0}^u(-1)^{u-t}\binom{u}{t}T_{2 u+t}\nonumber\\
&=\alpha^{-i}\sum_{t=0}^j(-1)^{j-t}\binom{j}{t}\sum_{u=t}^j\binom{j-t}{u-t}T_{2 u+t}\,.
\end{align}
\end{Cor}

In particular, when $u=v=w=1$, which means $\alpha=\beta=1$, for $0\leq{i}\leq{k}$, putting $x_i:=\sum_{j=0}^{n-i}\binom{i}{j}T_{n-(i+j)}$, (\ref{mat1}) becomes a nonhomogeneous linear system of equations with ($k+1$) variables:
$$
\left\{
\begin{array}{l}\displaystyle
\binom{k}{0}x_0+\binom{k}{1}x_1+\cdots+\binom{k}{k-1}x_{k-1}+\binom{k}{k}x_k=T_{n+k}\,, \\\displaystyle
\binom{k-1}{0}x_0+\binom{k-1}{1}x_1+\cdots+\binom{k-1}{k-1}x_{k-1}+0\cdot x_k=T_{n+k-1}\,, \\\displaystyle
\dots\\\displaystyle
\binom{1}{0}x_0+\binom{1}{1}x_1+0\cdot x_2+\cdots++0\cdot x_k=T_{n+1}\,, \\\displaystyle
\binom{0}{0}x_0+0\cdot x_1+0\cdot x_2+\cdots+0\cdot x_k=T_{n}\,
\end{array}
\right.
$$
Let $B_{i}$ be the matrix obtained from $\widetilde{P}$ by replacing its $i$-th column with the vector $\mathbf{T}$. 
Since $\widetilde{P}$ is the $(k+1)\times(k+1)$ lower-triangular Pascal matrix whose lower right half elements are all zero and anti-diagonal entries are all $1$, we get $|\widetilde{P}|=(-1)^{\binom{k+1}{2}}$.  
By Cramer's rule, we obtain
\begin{align*}
&\sum_{j=0}^{n-i}\binom{i}{j}T_{n-(i+j)}\\
&=\frac{|B_i|}{|\widetilde{P}|}
=(-1)^{\frac{k(k+1)}{2}}
\left|
\begin{array}{ccccccc}
\binom{k}{0} & \cdots & \binom{k}{i-1} & T_{n+k} & \binom{k}{i+1} & \cdots & \binom{k}{k} \\
\binom{k-1}{0} & \cdots & \binom{k-1}{i-1} & T_{n+k-1} & \binom{k-1}{i+1} & \cdots & 0 \\
\vdots & & \vdots & \vdots & \vdots & & \vdots \\
\binom{1}{0} & \cdots & 0 & T_{n+1} & 0 & \cdots & 0 \\
\binom{0}{0} & \cdots & 0 & T_{n} & 0 & \cdots & 0
\end{array}
\right|\,.
\end{align*}
Note that the left-hand side does not depend on $k$. 

Since 
\begin{equation}
|B_i|
=(-1)^{i}
\left|
\begin{array}{ccccccc}
T_{n+k} & \binom{k}{0} & \cdots &\binom{k}{i-1} & \binom{k}{i+1} & \cdots & \binom{k}{k}\\
T_{n+k-1} & \binom{k-1}{0} & \cdots & \binom{k-1}{i-1} & \binom{k-1}{i+1} & \cdots & 0 \\
\vdots & \vdots & & \vdots & \vdots & & \vdots \\
T_{n+i+1} & \binom{i+1}{0} & \cdots & \binom{i+1}{i-1} & \binom{i+1}{i+1} & \cdots & 0 \\
T_{n+i} & \binom{i}{0} & \cdots & \binom{i}{i-1} & 0 & \cdots & 0 \\
T_{n+i-1} & \binom{i-1}{0} & \cdots & \binom{i-1}{i-1} & 0 & \cdots & 0 \\
\vdots & \vdots & & \vdots & \vdots & & \vdots \\
T_{n} & \binom{0}{0} & \cdots & 0 & 0 & \cdots &0
\end{array}
\right|\,, 
\label{det-rex1}
\end{equation}  
we have the following expression.    

\begin{Cor}
\begin{align*}
&\sum_{j=0}^{n-i}\binom{i}{j}T_{n-(i+j)}\\
&=(-1)^{\frac{k(k+1)}{2}+i}
\left|
\begin{array}{ccccccc}
T_{n+k} & \binom{k}{0} & \cdots &\binom{k}{i-1} & \binom{k}{i+1} & \cdots & \binom{k}{k}\\
T_{n+k-1} & \binom{k-1}{0} & \cdots & \binom{k-1}{i-1} & \binom{k-1}{i+1} & \cdots & 0 \\
\vdots & \vdots & & \vdots & \vdots & & \vdots \\
T_{n+i+1} & \binom{i+1}{0} & \cdots & \binom{i+1}{i-1} & \binom{i+1}{i+1} & \cdots & 0 \\
T_{n+i} & \binom{i}{0} & \cdots & \binom{i}{i-1} & 0 & \cdots & 0 \\
T_{n+i-1} & \binom{i-1}{0} & \cdots & \binom{i-1}{i-1} & 0 & \cdots & 0 \\
\vdots & \vdots & & \vdots & \vdots & & \vdots \\
T_{n} & \binom{0}{0} & \cdots & 0 & 0 & \cdots &0
\end{array}
\right|\,.
\end{align*}
\end{Cor}

The usage of matrix approaches and binomial inversion gives a series of identities about Tribonacci numbers, systematically giving a mutual expression between later terms and previous terms. It may be suitable for deriving symmetric identities for linear recurrence sequences.

\section{Tribonacci determinants}

The generating function of generalized Tribonacci number $\mathcal T_n^{(a,b,c)}$ is given by
\begin{equation}
\mathcal G(t):=\sum_{n=0}^\infty\mathcal T_n^{(a,b,c)}t^n=\frac{a+(b-u a)t+(c-u b-v a)t^2}{1-u t-v t^2-w t^3}\,.
\label{gf:ggtribo}
\end{equation}
When $a=b=c=1$, it is reduced to that of the classical Tribonacci numbers, given by
$$
\sum_{n=0}^\infty T_n t^n=\frac{t}{1-t-t^2-t^3}\,.
$$

In \cite{KY26}, the determinant expressions of the generalized Tribonacci numbers $\mathcal T_{n+1}^{(0,1,1)}$ and $\mathcal T_{2 n+2}^{(0,1,1)}$ were given, though that of $\mathcal T_{2 n+1}^{(0,1,1)}$ was unknown.

\subsection{A determinant expression of $\mathcal T_{2 n+1}^{(0,1,1)}$}

In this subsection, we shall show a determinant expression of $\mathcal T_{2 n+1}^{(0,1,1)}$, which could not be obtained in \cite{KY26}.  For simplicity, put $\mathcal T_{2 n+1}:=\mathcal T_{2 n+1}^{(0,1,1)}$ in this subsection.

Put
$$
\gamma:=\frac{D_1+\sqrt{D_1^2+4 D_2}}{2}\quad\hbox{and}\quad\delta:=\frac{D_1-\sqrt{D_1^2+4 D_2}}{2}\,,
$$
where $D_1=u^2-u+v$ and $D_2=w(u-1)$.  For $n\ge 2$, let
\begin{equation}
r_n=\frac{r_3-r_2\delta}{\gamma-\delta}\gamma^{n-2}+\frac{r_2\gamma-r_3}{\gamma-\delta}\delta^{n-2}
\label{r-binet}
\end{equation}
with $r_1=-(u+v)$, $r_2=u^2-u^3-u^2 v-u w-w$ and $r_3=D_1 r_2+D_2 r_1-w^2$.

\begin{theorem}
For $n\ge 1$,
$$
\mathcal T_{2 n+1}=\left|\begin{array}{ccccc}
-r_1&1&0&\cdots&0\\
r_2&-r_1&\ddots&&\vdots\\
-r_3&r_2&&\ddots&0\\
\vdots&&&-r_1&1\\
(-1)^n r_n&\cdots&-r_3&r_2&-r_1
\end{array}\right|\,.
$$
\label{th:det-t2n1}
\end{theorem}

By using the inversion relation in \cite[Theorem 1, Lemma 4]{KR18}:
$$
\alpha_n=\begin{vmatrix} R(1) & 1 & & 0\\
R(2) & \ddots &  \ddots & \\
\vdots & \ddots &  \ddots & 1\\
R(n) & \cdots &  R(2) & R(1) \\
 \end{vmatrix}\quad
\Longleftrightarrow\quad
R(n)=\begin{vmatrix} \alpha_1 & 1 & &0 \\
\alpha_2 & \ddots &  \ddots & \\
\vdots & \ddots &  \ddots & 1\\
\alpha_n & \cdots &  \alpha_2 & \alpha_1 \\
 \end{vmatrix}\,,
$$
we obtain the following result.

\begin{Cor}
For $n\ge 1$,
$$
\left|\begin{array}{ccccc}
\mathcal T_{3}&1&0&\cdots&0\\
\mathcal T_{5}&\mathcal T_{3}&\ddots&&\vdots\\
\mathcal T_{7}&\mathcal T_{5}&&\ddots&0\\
\vdots&&&\mathcal T_{3}&1\\
\mathcal T_{2 n+1}&\cdots&\mathcal T_{7}&\mathcal T_{5}&\mathcal T_{3}
\end{array}\right|=
(-1)^n r_n\,.
$$
\label{cor:det-t2n1}
\end{Cor}

\noindent
{\it Remark.}
When $u=v=w=1$, by $r_1=-2$, $r_2=-3$, $r_n=-4$ ($n\ge 3$), the determinant in Theorem \ref{th:det-t2n1} is reduced to
$$
T_{2 n+1}=\left|\begin{array}{ccccc}
2&1&0&\cdots&0\\
-3&2&\ddots&&\vdots\\
4&-3&&\ddots&0\\
\vdots&&&2&1\\
4(-1)^{n-1}&\cdots&4&-3&2
\end{array}\right|\,.
$$
(see, \cite[p.1032]{KY26}).
When $u=v=w=1$, the determinant in Corollary \ref{cor:det-t2n1} is reduced to
$$
\left|\begin{array}{ccccc}
T_{3}&1&0&\cdots&0\\
T_{5}&T_{3}&\ddots&&\vdots\\
T_{7}&T_{5}&&\ddots&0\\
\vdots&&&T_{3}&1\\
T_{2 n+1}&\cdots&T_{7}&T_{5}&T_{3}
\end{array}\right|=\begin{cases}
2&\text{if $n=1$};\\
-3&\text{if $n=2$};\\
4(-1)^{n-1}&\text{if $n\ge 3$}
\end{cases}
$$
(\cite[Theorem 3.2]{GS20}).
\bigskip

In order to prove Theorem \ref{th:det-t2n1}, we need several lemmas.

\begin{Lem}
For $n\ge 6$
$$
\mathcal T_n=(u^2+2 v)\mathcal T_{n-2}-(v^2-2 u w)\mathcal T_{n-4}+w^2 T_{n-6}\,.
$$
\label{lem:rel-2step}
\end{Lem}
\begin{proof}
By the definition in (\ref{def:ggtribo}), from
$$
w\mathcal T_{n-3}=u w\mathcal T_{n-4}+v w\mathcal T_{n-5}+w^2\mathcal T_{n-6}
$$
and
$$
v\mathcal T_{n-2}=u v\mathcal T_{n-3}+v^2\mathcal T_{n-4}+v w\mathcal T_{n-5}\,,
$$
we get
\begin{equation}
w\mathcal T_{n-3}=u w\mathcal T_{n-4}+v\mathcal T_{n-2}-u v\mathcal T_{n-3}-v^2\mathcal T_{n-4}+w^2\mathcal T_{n-6}\,.
\label{eq:gt34}
\end{equation}
By the definition in (\ref{def:ggtribo}) again and (\ref{eq:gt34}), we have
\begin{align*}
\mathcal T_n&=u\mathcal T_{n-1}+v\mathcal T_{n-2}+w\mathcal T_{n-3}\\
&=u(u\mathcal T_{n-2}+v\mathcal T_{n-3}+w\mathcal T_{n-4})+v\mathcal T_{n-2}\\
&\qquad +u w\mathcal T_{n-4}+v\mathcal T_{n-2}-u v\mathcal T_{n-3}-v^2\mathcal T_{n-4}+w^2\mathcal T_{n-6}\\
&=(u^2+2 v)\mathcal T_{n-2}-(v^2-2 u w)\mathcal T_{n-4}+w^2 T_{n-6}\,.
\end{align*}
\end{proof}

The generating function of $\mathcal T_{2 n+1}$ is given as follows.

\begin{Lem}
$$
\sum_{n=0}^\infty\mathcal T_{2 n+1}t^n=\frac{1-(u^2-u+v)t-w(u-1)t^2}{1-(u^2+2 v)t+(v^2-2 u w)t^2-w^2 t^3}\,.
$$
\label{lem:gf-ggtribo}
\end{Lem}
\begin{proof}
Put $f(t):=\sum_{n=0}^\infty\mathcal T_{2 n+1}t^n$. By combining this with $t f(t)$, $t^2 f(t)$ and $t^3 f(t)$, and using Lemma \ref{lem:rel-2step} to eliminate higher-order infinite terms, we obtain
\begin{align*}
&\bigl(1-(u^2+2 v)t+(v^2-2 u w)t^2-w^2 t^3\bigr)f(t)\\
&=\mathcal T_1+\bigl(\mathcal T_3-(u^2+2 v)\mathcal T_1\bigr)t+\bigl(\mathcal T_5-(u^2+2 v)\mathcal T_3+(v^2-2 u w)T_1\bigr)t^2\\
&=1-(u^2-u+v)t-w(u-1)t^2\,.
\end{align*}
\end{proof}

The sequence $r_n$ arises from the reciprocal of the series in Lemma \ref{lem:gf-ggtribo}.

\begin{Lem}
$$
\left(\sum_{n=0}^\infty\mathcal T_{2 n+1}t^n\right)^{-1}=1+\sum_{n=1}^\infty r_n t^n\,.
$$
\label{lem:gf-inv-ggtribo}
\end{Lem}
\begin{proof}
Since
\begin{align*}
1&=\left(\sum_{n=0}^\infty\mathcal T_{2 n+1}t^n\right)\left(\sum_{n=0}^\infty r_n t^n\right)\\
&=\sum_{n=0}^\infty\left(\sum_{k=0}^n\mathcal T_{2(n-k)+1}r_k\right)t^n
\end{align*}
($r_0=1$ for convenience),  we have for $n\ge 1$
$$
\sum_{k=0}^n\mathcal T_{2(n-k)+1}r_k=0
$$
or
\begin{equation}
r_n=-\sum_{k=0}^{n-1}\mathcal T_{2(n-k)+1}r_k\quad(n\ge 1)\,.
\label{eq:rel89}
\end{equation}
When $n=1$ in (\ref{eq:rel89}), $r_1=-\mathcal T_3=-(u+v)$.
When $n=2$,
\begin{align*}
r_2&=-(\mathcal T_5+\mathcal T_3 r_1)\\
&=-(u^3+2 u v+u^2 v+v^2+w+u w)+(u+v)^2\\
&=u^2-u^3-u^2 v-u w-w\,.
\end{align*}
When $n=3$,
\begin{align*}
r_3&=-(\mathcal T_7+\mathcal T_5 r_1+\mathcal T_3 r_2)\\
&=D_1 r_2+D_2 r_1-w^2\,,
\end{align*}
where $D_1=u^2-u+v$ and $D_2=w(u-1)$.
We shall prove that for $n\ge 4$
\begin{equation}
r_n=D_1 r_{n-1}+D_2 r_{n-2}\,.
\label{eq:rn4}
\end{equation}
It is easily checked that the cases for $n=4,5$ are valid.
Suppose that the relation (\ref{eq:rn4}) is valid for $n=m-1$ and $n=m-2$.
Then, by using the assumption and Lemma \ref{lem:rel-2step},
\begin{align*}
r_n&=-D_1(\mathcal T_{2 n-1}+r_1\mathcal T_{2 n-3}+r_2\mathcal T_{2 n-5}+r_3\mathcal T_{2 n-7}+\cdots+r_{n-2}\mathcal T_3)\\
&\qquad -D_2(\mathcal T_{2 n-3}+r_1\mathcal T_{2 n-5}+r_2\mathcal T_{2 n-7}+\cdots+r_{n-2}\mathcal T_3)\\
&=-\bigl(D_1 T_{2 n-1}+(D_1 r_1+D_2)T_{2 n-3}+(w^2+r_3)T_{2 n-5}\bigr)\\
&\qquad -(r_4 T_{2 n-7}+\cdots+r_{n-1}T_3)\\
&=-(T_{2 n+1}+r_1 T_{2 n-1}+r_2 T_{2 n-3}+r_3 T_{2 n-5}+r_4 T_{2 n-7}+\cdots+r_{n-1}T_3)\,.
\end{align*}
Hence, the relation (\ref{eq:rn4}) is also valid for $n=m$.
Solving (\ref{eq:rn4}), we get the expression of $r_n$ in (\ref{r-binet}).
\end{proof}

Finally, the result in Theorem \ref{th:det-t2n1} is due to Cameron's operator $\mathcal A$, which is defined on the set of sequences of non-negative integers as follows: for $x=\{x_n\}_{n\ge 1}$ and $z=\{z_n\}_{n\ge 1}$, set $\mathcal A x=z$, where
\begin{equation}
1+\sum_{n=1}^\infty z_n t^n=\left(1-\sum_{n=1}^\infty x_n t^n\right)^{-1}
\label{cameron}
\end{equation}
(\cite{Cameron}). It is shown in \cite{KKL22,Ko20b} that the Cameron's operator corresponds to the determinant expression for its explanation and applications.

\begin{Lem}
For an integer $n\ge 1$,
$$
z_n=\left|
\begin{array}{ccccc}
x_1&1&0&&\\
-x_2&x_1&&&\\
\vdots&\vdots&\ddots&1&0\\
(-1)^n x_{n-1}&(-1)^{n-1}x_{n-2}&\cdots&x_1&1\\
(-1)^{n-1}x_n&(-1)^{n}x_{n-1}&\cdots&-x_2&x_1
\end{array}
\right|\,.
$$
\label{lem:cameron}
\end{Lem}

Theorem \ref{th:det-t2n1} is the direct consequence of Lemma \ref{lem:cameron}.

\subsection{More on Tribonacci determinants}

In this subsection, we aim not merely to present determinant representations of Tribonacci numbers, but to demonstrate that Tribonacci numbers with specific initial values correspond to those with different initial values in terms of determinants, and, furthermore, that such correspondence is highly restricted.

In the definition (\ref{def:bell-p}) of the complete Bell polynomials, we choose $x_m=(m-1)!\mathcal T_m^{(a,b,c)}$.
Let
$$
\mathcal F(t):=\sum_{m=1}^\infty (m-1)!\mathcal T_m^{(a,b,c)}\frac{t^m}{m!}
=\sum_{m=1}^\infty\frac{\mathcal T_m^{(a,b,c)}}{m}t^m\,.
$$
Then,
\begin{align*}
\mathcal F'(t)&=\frac{1}{t}\sum_{m=1}^\infty\mathcal T_m^{(a,b,c)}t^m=\frac{\mathcal G(t)-a}{t}\\
&=\frac{1}{t}\left(\frac{a+(b-u a)t+(c-u b-v a)t^2}{1-u t-v t^2-w t^3}-a\right)\\
&=\frac{b+(c-u b)t+w a t^2}{1-u t-v t^2-w t^3}\,.
\end{align*}
Since
$$
\frac{d}{d t}\log(1-u t-v t^2-w t^3)=-\frac{u+2 v t+3 w t^2}{1-u t-v t^2-w t^3}\,,
$$
when $(a,b,c)=(3,u,u^2+2 v)$, we have
\begin{align*}
\exp\left(\mathcal F(t)\right)&=\exp\left(-\log(1-u t-v t^2-w t^3)\right)\\
&=\frac{1}{1-u t-v t^2-w t^3}
=\sum_{n=0}^\infty\mathcal T_{n+1}^{(0,1,u)}t^n\,.
\end{align*}
Notice that if $(a,b,c)\ne(3,u,u^2+2 v)$, then the form of $\mathcal F(t)$ is not simple, nor is $\exp\left(\mathcal F(t)\right)$.

Therefore, by Lemma \ref{lem:gtrudi} (1) and (2), we obtain

\begin{theorem}
\begin{align*}
\mathcal T_{n+1}^{(0,1,u)}&=\frac{1}{n!}\mathbf Y_n(\mathcal T_1^{(3,u,u^2+2 v)},1!\mathcal T_2^{(3,u,u^2+2 v)},2!\mathcal T_3^{(3,u,u^2+2 v)},3!\mathcal T_4^{(3,u,u^2+2 v)},\dots)\\
&=\frac{1}{n!}\left|\begin{array}{ccccc}
\mathcal T_1^{(3,u,u^2+2 v)}&-1&0&\cdots&0\\
\mathcal T_2^{(3,u,u^2+2 v)}&\mathcal T_1^{(3,u,u^2+2 v)}&-2&&\vdots\\
\vdots&&\ddots&&0\\
\mathcal T_{n-1}^{(3,u,u^2+2 v)}&\mathcal T_{n-2}^{(3,u,u^2+2 v)}&\cdots&\mathcal T_1^{(3,u,u^2+2 v)}&-n+1\\
\mathcal T_{n}^{(3,u,u^2+2 v)}&\mathcal T_{n-1}^{(3,u,u^2+2 v)}&\cdots&\mathcal T_2^{(3,u,u^2+2 v)}&\mathcal T_1^{(3,u,u^2+2 v)}
\end{array}
\right|\,.
\end{align*}
\end{theorem}

By the inversion formula in Lemma \ref{lem:gtrudi} (2) and (3), we have

\begin{Cor}
$$
\mathcal T_{n}^{(3,u,u^2+2 v)}=(-1)^{n-1}\left|\begin{array}{ccccc}
\mathcal T_2^{(0,1,u)}&1&0&\cdots&0\\
2\mathcal T_3^{(0,1,u)}&\mathcal T_2^{(0,1,u)}&1&&\vdots\\
\vdots&&\ddots&&0\\
(n-1)\mathcal T_{n}^{(0,1,u)}&\mathcal T_{n-1}^{(0,1,u)}&\cdots&\mathcal T_2^{(0,1,u)}&1\\
n\mathcal T_{n+1}^{(0,1,u)}&\mathcal T_{n}^{(0,1,u)}&\cdots&\mathcal T_3^{(0,1,u)}&\mathcal T_2^{(0,1,u)}
\end{array}
\right|\,.
$$
\end{Cor}

\noindent
{\it Remark.}
When $(u,v,w)=(1,1,1)$, we have
\begin{align*}
T_{n+1}&=\frac{1}{n!}\mathbf Y_n(T_1^{(3,1,3)},1!T_2^{(3,1,3)},2!T_3^{(3,1,3)},3!T_4^{(3,1,3)},\dots)\\
&=\frac{1}{n!}\left|\begin{array}{ccccc}
T_1^{(3,1,3)}&-1&0&\cdots&0\\
T_2^{(3,1,3)}&T_1^{(3,1,3)}&-2&&\vdots\\
\vdots&&\ddots&&0\\
T_{n-1}^{(3,1,3)}&T_{n-2}^{(3,1,3)}&\cdots&T_1^{(3,1,3)}&-n+1\\
T_{n}^{(3,1,3)}&T_{n-1}^{(3,1,3)}&\cdots&T_2^{(3,1,3)}&T_1^{(3,1,3)}\\
\end{array}
\right|\,.
\end{align*}
By the inversion formula, we have
$$
T_{n}^{(3,1,3)}=(-1)^{n-1}\left|\begin{array}{ccccc}
T_2&1&0&\cdots&0\\
2 T_3&T_2&1&&\vdots\\
\vdots&&\ddots&&0\\
(n-1)T_{n}&T_{n-1}&\cdots&T_2&1\\
n T_{n+1}&T_{n}&\cdots&T_3&T_2\\
\end{array}
\right|\,.
$$
We remark that the sequence $T_{n}^{(3,1,3)}$ appears in \cite[A001644]{oeis} and is given as
$$
\{T_{n}^{(3,1,3)}\}_{n\ge 0}=3, 1, 3, 7, 11, 21, 39, 71, 131, 241, 443, 815, 1499, 2757, 5071, \dots\,.
$$
This sequence is also known as Tribonacci-Lucas sequence (\cite{RS15}) or $3$-step Lucas sequence (\cite[Definition 13]{SW11}).
\medskip

One might think that the results mentioned above can be easily modified for other Tribonacci numbers with different initial values, but the structure is not that simple.  This also implies that there is a strong relation between the original Tribonacci numbers $T_n=\mathcal T_n^{(0,1,1)}$ and the Tribonacci-Lucas numbers $\mathcal T_n^{(3,1,3)}$ like that of Fibonacci numbers and Lucas numbers.

\subsection{Bell polynomials}

The (complete exponential) Bell polynomial $\mathbf Y_n(x_1,x_2,\dots,x_n)$ is defined by
\begin{equation}
\exp\left(\sum_{m=1}^\infty x_m\frac{t^m}{m!}\right)=1+\sum_{n=1}^\infty\mathbf Y_n(x_1,x_2,\dots,x_n)\frac{t^n}{n!}\,.
\label{def:bell-p}
\end{equation}
It is expressed as
\begin{align*}
&\mathbf Y_n(x_1,x_2,x_3,\dots,x_n)\\
&=\sum_{k=1}^n\sum_{i_1+2 i_2+\cdots+(n-k+1)i_{n-k+1}=n\atop i_1+i_2+i_3+\cdots=k}\frac{n!}{i_1!i_2!\cdots i_{n-k+1}!}\\
&\qquad\qquad\times \left(\frac{x_1}{1!}\right)^{i_1}\left(\frac{x_2}{2!}\right)^{i_2}\cdots\left(\frac{x_{n-k+1}}{(n-k+1)!}\right)^{i_{n-k+1}}
\end{align*}
with $\mathbf Y_0=1$ (see, e.g., \cite[\S 3.3]{Comtet}).
In \cite{Ko26m1} and the related papers \cite{CKm,Ko25m2,KLm5,KPm4,KWm3}, by applying the theory of Bell polynomials (some fundamental theories can be found in \cite{MacDonald95}), we explicitly expressed the $q$-multiple zeta values specifically in cases involving general indices (powers) and a small number of summands using determinants.

For two sequences $a_0=1,a_1,a_2,\dots$ and $b_0=1,b_1,b_2,\dots$, we have the equivalent expressions \cite{Ko25m2}. Note that a different version can be seen in \cite[Lemma 5]{Ko26m1}.

\begin{Lem}
The following expressions are equivalent.
\begin{enumerate}
\item[(1)] $\displaystyle b_m=\frac{1}{m!}\mathbf Y_m\bigl(a_1,1!a_2,2!a_3,3!a_4,\dots,(m-1)!a_m)$
\item[(2)] $\displaystyle b_m=\frac{1}{m!}\left|\begin{array}{ccccc}
a_1&-1&0&\cdots&0\\
a_2&a_1&-2&&\vdots\\
\vdots&&\ddots&&0\\
a_{m-1}&a_{m-2}&\cdots&a_1&-m+1\\
a_m&a_{m-1}&\cdots&a_2&a_1\\
\end{array}
\right|$
\item[(3)] $\displaystyle a_n=(-1)^{n-1}\left|\begin{array}{ccccc}
b_1&1&0&\cdots&0\\
2 b_2&b_1&1&&\vdots\\
\vdots&&\ddots&&0\\
(n-1)b_{n-1}&b_{n-2}&\cdots&b_1&1\\
n b_n&b_{n-1}&\cdots&b_2&b_1\\
\end{array}
\right|$
\end{enumerate}
\label{lem:gtrudi}
\end{Lem}

\subsection{Fibonacci $\ell$-step numbers}

The results achieved with the aforementioned Tribonacci numbers may potentially be replicated with sequences involving a greater number of terms, that is, the sequences of Tetranacci, Pentanacci, and hexanacci numbers.
For the sake of simplicity, the discussion shall be limited to cases satisfying a recurrence relation in which all coefficients are equal to one.

An $\ell$-step Fibonacci sequence $\{F_n^{(\ell)}\}_{n\ge 0}$ is defined by
\begin{equation}
F_n^{(\ell)}=F_{n-1}^{(\ell)}+F_{n-2}^{(\ell)}+\cdots+F_{n-\ell}^{(\ell)}\quad(n\ge 3)
\label{def:fibostep}
\end{equation}
with $F_n^{(\ell)}=0$ ($n\le 0$) and $F_1^{(\ell)}=F_2^{(\ell)}=1$.  More generally, if the sequence $\{a_n\}_n$ satisfies the recurrence relation
\begin{equation}
a_n=a_{n-1}+a_{n-2}+\cdots+a_{n-\ell}
\label{def:gfibostep}
\end{equation}
with arbitrary initial terms, then its generating function is given by
\begin{equation}
\mathfrak G(t)=\frac{a_0+(a_1-a_0)t+(a_2-a_1-a_0)t^2+\cdots+(a_{\ell-1}-a_{\ell-2}-\cdots-a_0)t^{\ell-1}}{1-t-t^2-\cdots-t^{\ell}}\,.
\label{gf:gfibostep}
\end{equation}

By setting
$$
\mathfrak F(t):=\sum_{m=1}^\infty (m-1)!a_m\frac{t^m}{m!}
=\sum_{m=1}^\infty\frac{a_m}{m}t^m\,,
$$
we get
\begin{align*}
&\mathfrak F'(t)\\
&=\frac{1}{t}\sum_{m=1}^\infty a_m t^m=\frac{\mathfrak G(t)-a_0}{t}\\
&=\frac{1}{t}\left(\frac{a_0+(a_1-a_0)t+(a_2-a_1-a_0)t^2+\cdots+(a_{\ell-1}-a_{\ell-2}-\cdots-a_0)t^{\ell-1}}{1-t-t^2-\cdots-t^{\ell}}-a_0\right)\\
&=\frac{a_1+(a_2-a_1)t+(a_3-a_2-a_1)t^2+\cdots+(a_{\ell-1}-a_{\ell-2}-\cdots-a_1)t^{\ell-2}+a_0 t^{\ell-1}}{1-t-t^2-\cdots-t^{\ell}}\,.
\end{align*}
Since
$$
\frac{d}{d t}\log(1-t-t^2-\cdots-t^{\ell})=-\frac{1+2 t+3 t^2+\cdots+\ell t^{\ell-1}}{1-t-t^2-\cdots-t^{\ell}}\,,
$$
when $a_0=\ell$ and $a_j=2^j-1$ ($1\le j\le\ell-1$), we have
\begin{align*}
\exp\left(\mathcal F(t)\right)&=\exp\left(-\log(1-t-t^2-\cdots-t^{\ell})\right)\\
&=\frac{1}{1-t-t^2-\cdots-t^{\ell}}
=\sum_{n=0}^\infty F_{n+1}^{(\ell)}t^n\,.
\end{align*}

For simplicity, denote the sequence of $a_n$ in (\ref{def:gfibostep}) with initial values $a_0=\ell$ and $a_j=2^j-1$ ($1\le j\le\ell-1$) by $\mathfrak F_n$.
In conclusion, we obtain

\begin{theorem}
\begin{align*}
F_{n+1}^{(\ell)}&=\frac{1}{n!}\mathbf Y_n(\mathfrak F_1,1!\mathfrak F_2,2!\mathfrak F_3,3!\mathfrak F_4,\dots)\\
&=\frac{1}{n!}\left|\begin{array}{ccccc}
\mathfrak F_1&-1&0&\cdots&0\\
\mathfrak F_2&\mathfrak F_1&-2&&\vdots\\
\vdots&&\ddots&&0\\
\mathfrak F_{n-1}&\mathfrak F_{n-2}&\cdots&\mathfrak F_1&-n+1\\
\mathfrak F_{n}&\mathfrak F_{n-1}&\cdots&\mathfrak F_2&\mathfrak F_1\\
\end{array}
\right|\,.
\end{align*}
\end{theorem}

By the inversion formula, we have

\begin{Cor}
$$
\mathfrak F_n=(-1)^{n-1}\left|\begin{array}{ccccc}
F_2^{(\ell)}&1&0&\cdots&0\\
2 F_3^{(\ell)}&F_2^{(\ell)}&1&&\vdots\\
\vdots&&\ddots&&0\\
(n-1)F_n^{(\ell)}&F_{n-1}^{(\ell)}&\cdots&F_2^{(\ell)}&1\\
n F_{n+1}^{(\ell)}&F_n^{(\ell)}&\cdots&F_3^{(\ell)}&F_2^{(\ell)}\\
\end{array}
\right|\,.
$$
\end{Cor}



The work of T.K. was partly supported by JSPS KAKENHI Grant Number 24K22835.


\medskip
\noindent
{\bf Availability of data and materials}\\\noindent
Not applicable.



\begin{thebibliography}{99}

\bibitem{Cameron}
P. J. Cameron, {\em
Some sequences of integers},
Discret. Math. {\bf 75} (1989), 89--102.

\bibitem{CKm}
T. Chatterjee and T. Komatsu, {\em
Special values of a $q$-multiple t-function of general level at roots of unity},
Mediterr. J. Math. {\bf 22} (2025), No.8, Article 226, 20 p.

\bibitem{Comtet}
L. Comtet, {\em
Advanced Combinatorics},
Dordrecht, D. Reidel, 1974.

\bibitem{GS20}
T. Goy and M. Shattuck, {\em
Determinant identities for Toeplitz--Hessenberg matrices with tribonacci entries},Trans. Comb. {\bf 9} (2020), 89--109.

\bibitem{KKL22}
N. R. Kanasri, T. Komatsu and V. Laohakosol, {\em
Cameron's operator in terms of determinants and hypergeometric numbers},
Bol. Soc. Mat. Mex. {\bf 28} (2022), Paper No.9, 23 p.

\bibitem{KT10}
E. K{\i}l{\i}\c{c} and D. Ta\c{s}c{\i}, {\em
On the generalized Fibonacci and Pell sequences by Hessenberg matrices},
Ars Comb. {\bf 94} (2010), 161--174.

\bibitem{KTH10}
E. K{\i}l{\i}\c{c}, D. Ta\c{s}c{\i} and P. Haukkanen, {\em
On the generalized Lucas sequences by Hessenberg matrices},
Ars Comb. {\bf 95} (2010), 383--395.

\bibitem{Ko20}
T. Komatsu, {\em
Continued fraction expansions of the generating functions of Bernoulli and related numbers},
Indag. Math., New Ser. {\bf 31} (2020), No.4, 695--713.

\bibitem{Ko20b}
T. Komatsu, {\em
Fibonacci determinants with Cameron's operator},
Bol. Soc. Mat. Mex. {\bf 26} (2020), 841--863.

\bibitem{Ko21}
T. Komatsu, {\em
Several continued fraction expansions of generalized Cauchy numbers},
Bull. Malays. Math. Sci. Soc. (2) {\bf 44} (2021), No.4, 2425--2446.

\bibitem{Ko25a}
T. Komatsu, {\em
Fibonacci determinants, II},
Lith. Math. J. {\bf 65} (2025), No. 4, 564--584.
DOI: 10.1007/s10986-025-09691-1

\bibitem{Ko25m2}
T. Komatsu, {\em
Some explicit values of a $q$-multiple zeta-star function at roots of unity},
RAIRO, Theor. Inform. Appl. {\bf 59} (2025), Paper No.5, 14 p.

\bibitem{Ko26m1}
T. Komatsu, {\em
Some explicit values of a $q$-multiple zeta function at roots of unity},
J. Math. Anal. Appl. {\bf 555} (2026), No.2, Article 130065, 16 p.

\bibitem{KLm5}
T. Komatsu and F. Luca, {\em
Some explicit forms of special values of an alternating $q$-multiple $t$-function of general level at roots of unity},
Ramanujan J. {\bf 68} (2025), Article no.3, 25 p.

\bibitem{KPm4}
T. Komatsu and R. K. Pandey, {\em
Some explicit values of an alternative q-multiple zeta function at roots of unity},
Log. J. IGPL {\bf 33} (2025), No.6, Article ID jzaf075, 19 p.

\bibitem{KR18}
T. Komatsu and J. L. Ramirez, {\em
Some determinants involving incomplete Fubini numbers},
An. \c{S}tiin\c{t}. Univ. ``Ovidius'' Constan\c{t}a Ser. Mat. {\bf 26} (2018), No.3, 143--170.

\bibitem{KWm3}
T. Komatsu and T. Wang, {\em
Some explicit values of a q-multiple t-function at roots of unity},
Aequationes Math. {\bf 99} (2025), No. 5, 2377--2400.
DOI:10.1007/s00010-025-01235-9

\bibitem{KY26}
T. Komatsu and F. Yilmaz, {\em
New determinant expressions of Bernoulli, Euler and Tribonacci polynomials},
AIMS Math. {\bf 11} (2026), No. 1, 1021--1035.

\bibitem{Koshy1}
T. Koshy, {\em Fibonacci and Lucas Numbers with Applications},
Wiley, New York, 2001.

\bibitem{Koshy2}
T. Koshy, {\em
Fibonacci and Lucas numbers with applications. Vol. 2},
Pure Appl. Math.,
John Wiley \& Sons, Inc., Hoboken, NJ, 2019.

\bibitem{MacDonald95}
I. G. MacDonald, {\em
Symmetric Functions and Hall Polynomials},
2nd ed., Clarendon Press, Oxford, 1995.

\bibitem{RS15}
J. L. Ramirez and V. F. Sirvent, {\em
A generalization of the $k$-bonacci sequence from Riordan arrays},
Electron. J. Comb. {\bf 22} (2015), No. 1, Paper No. P1.38, 20 p.

\bibitem{SW11}
S. Saito, T Tanaka and N. Wakabayashi, {\em
Combinatorial remarks on the cyclic sum formula for multiple zeta values},
J. Int. Seq. {\bf 14} (2011), Article 11.2.4, 20 p.

\bibitem{oeis}
N. J. A. Sloane, {\em
The On-Line Encyclopedia of Integer Sequences},
available at oeis.org.  (2026).

\bibitem{TongronK22} 
Y. Tongron and T. Komatsu, {\em  
Explicit formulae for some sums of Trinomial coefficients},  
rog. Appl. Sci. Tech. (PAST) 10 (2020), 161--169. 
https://ph02.tci-thaijo.org/index.php/past/article/view/242916 

\bibitem{Trudi}
N. Trudi, {\em
Intorno ad alcune formole di sviluppo},
Rendic. dell' Accad. Napoli (1862), 135--143 (in Italian).

\end{thebibliography}
\end{document}